\documentclass[12pt]{article}

\usepackage{amssymb}
\usepackage{latexsym}
\usepackage{amsmath}
\usepackage{amsthm}
\usepackage{amscd}
\usepackage{enumerate}

\usepackage[abs]{overpic}

\newtheorem{theorem}{Theorem}[section]
\newtheorem{proposition}[theorem]{Proposition}
\newtheorem{corollary}[theorem]{Corollary}
\newtheorem{lemma}[theorem]{Lemma}
\newtheorem{lemma-definition}[theorem]{Lemma-Definition}

\newtheorem{conjecture}[theorem]{Conjecture}
\newtheorem{claim}[theorem]{Claim}

\theoremstyle{definition}
\newtheorem{question}[theorem]{Question}

\newtheorem{remark}[theorem]{Remark}

\newcommand{\op}{\operatorname}

\renewcommand{\epsilon}{\varepsilon}

\begin{document}

\title{Special eccentricities of rational four-dimensional ellipsoids}
\author{Dan Cristofaro-Gardiner}
\maketitle

\begin{abstract}

A striking result of McDuff and Schlenk asserts that in determining when a four-dimensional symplectic ellipsoid can be symplectically embedded into a four-dimensional symplectic ball, the answer is governed by an ``infinite staircase" determined by the odd-index Fibonacci numbers and the Golden Mean.  Here we study embeddings of one four-dimensional symplectic ellipsoid into another, and we show that if the target is rational, then the infinite staircase phenomenon found by McDuff and Schlenk is quite rare.  Specifically,  in the rational case, there is an infinite staircase in precisely three cases --- when the target has ``eccentricity" $1, 2$, or $3/2$; in all other cases the answer is given by the classical volume obstruction except on finitely many compact intervals on which it is linear.  This verifies in the special case of ellipsoids a conjecture by Holm, Mandini, Pires, and the author.  

\end{abstract}

\section{Introduction}

\subsection{The main theorem}

A {\em symplectic embedding} of one symplectic manifold $(M_1,\omega_1)$ into another $(M_2,\omega_2)$ is a smooth embedding
\[ \Psi: M_1 \to M_2 \]
such that $\Psi^* \omega_2 = \omega_1.$.   Determining whether or not a symplectic embedding exists can be very subtle, even in simple examples.  For example, define the (open) {\em symplectic ellipsoid}
\[ E(a_1, \ldots, a_n) := \left \lbrace \pi \frac{ |z_1|^2}{a_1} + \ldots + \pi \frac{ |z_n|^2}{a_n} < 1 \right \rbrace \subset \mathbb{C}^n = \mathbb{R}^{2n}. \]
This inherits a symplectic form by restricting the symplectic form on $\mathbb{R}^{2n}.$  Define the {\em symplectic ball} 
\[B^{2n}(\lambda) := E(\lambda, \ldots, \lambda).\] 

In \cite{ms}, McDuff and Schlenk determined exactly when a four-dimensional symplectic ellipsoid can be symplectically embedded into a four-dimensional symplectic ball.  Specifically, they computed the function
\[ c(a) := \op{min} \lbrace \lambda \hspace{1 mm} | E(1,a) \to B^4(\lambda) \rbrace\]
for $a \ge 1$,  where here and below the arrow denotes a symplectic embedding.  They found that the function $c(a)$ has a surprisingly rich structure:

\begin{theorem}\cite{ms}
\begin{itemize}

\item For $1 \le a \le \tau^4$, the function $c(a)$ is given by an infinite staircase determined by the odd-index Fibonacci numbers.
\item For $a \ge \left(\frac{17}{6}\right)^2$, we have $c(a) = \sqrt{a}$; in other words, the only obstruction to the embedding problem is the classical volume obstruction.
\item For $\tau^4 \le a \le \left (\frac{17}{6}\right)^2 $, we have $c(a) = \sqrt{a}$, except on finitely many intervals on which it is linear.

\end{itemize}
\end{theorem}

In fact, they compute the function precisely, see \cite{ms}, but we do not need their exact result here.

The purpose of this note is to show that, within the family of rational ellipsoids, the rigidity found by McDuff-Schlenk is in fact quite rare.  To make this precise, for fixed $b \ge 1$, define the function
\begin{equation}
\label{eqn:defn}
c_b(a) = \op{min} \lbrace \lambda | E(1,a) \to E(\lambda,\lambda b) \rbrace.
\end{equation}
Then, the function $c_1(a)$ is precisely the McDuff-Schlenk function considered above.  The function $c _b(a)$ is a continuous function, for example by \cite[Lem. 5.1]{cgk}, but is not in general $C^1$, as seen for example by the McDuff-Schlenk result above.   Following for example \cite{mcduffnote}, we call $b$ the {\em eccentricity} of the ellipsoid $E(1,b)$.    

We can now state our main result:

\begin{theorem}
\label{thm:main}
Fix a rational $b \ge 1$.  Then, unless $b \in \lbrace 1, 2, 3/2 \rbrace$, we have $c_b(a) = \sqrt{ \frac{a}{b} }$, except for finitely many compact intervals on which it is linear. 
\end{theorem}
 
Note that the quantity $\sqrt{\frac{a}{b}}$ represents the classical volume obstruction here.  (Symplectic embeddings must preserve volume.)  
 
 \begin{remark}

\begin{itemize}

\item In view of Theorem~\ref{thm:main}, it is natural to ask what is known about $c_b(a)$ when $b \in \lbrace 1, 2, 3/2 \rbrace$.  In fact, it was previously shown \cite{cgk} that in each of these cases, the function $c_b(a)$ starts with an infinite staircase, determined by an infinite sequence that generalizes the odd-index Fibonacci numbers.  So, from the point of view of infinite staircases for embeddings into rational ellipsoids, Theorem~\ref{thm:main} is an optimal result.   We also note that recently \cite{rr} have introduced a beautiful construction for explicitly constructing the embeddings required for the infinite staircase in the $b \in \lbrace 1, 2, 3/2 \rbrace$ case using almost toric fibrations.

\item For reasons related to the previous bullet point, Theorem~\ref{thm:main} was originally conjectured in \cite{cgk}, see \cite[Conj. 1.8]{cgk}.

\end{itemize}

\end{remark}

It is interesting to compare Theorem~\ref{thm:main} with a recent result of Usher \cite{usher}.  Usher studied an analogous function for embeddings into the {\em four-dimensional polydisc} $P(a,b) := D^2(a) \times D^2(b)$, namely he studied the function 
\[p_b(a) = \op{min} \lbrace \lambda | E(1,a) \to P(\lambda, \lambda b) \rbrace\] 
for fixed $b$. He found that for irrational $b$, there are infinitely many values of $b$ for which the function $p_b(a)$ has an infinite staircase.  This gives added intrigue to the following question, which is natural in view of Theorem~\ref{thm:main}.

\begin{question}
Are there irrational numbers $b$ for which $c_b(a)$ has infinitely many singular points? 
\end{question}

Here, by a {\em singular point}, we mean a value of $a$ where $c_b(a)$ is not differentiable.
 
\subsection{Reflexive polygons}
\label{sec:reflex}

Theorem~\ref{thm:main} verifies in a special case a recent conjecture of the author and Holm, Mandini, and Pires.

To explain this in more detail, we need to recall some terminology from (for example) \cite{concave, concaveconvex}.  Let $\Omega \subset \mathbb{R}^2$ be a region in the first quadrant.  We define the {\em toric domain} corresponding to $\Omega$ to be the subset
\[ X_{\Omega} = \lbrace (z_1, z_2) | (\pi |z_1|^2, \pi |z_2|^2 ) \in \Omega \rbrace \subset \mathbb{C}^n = \mathbb{R}^{2n},\]
with the symplectic form inherited from the standard from on $\mathbb{R}^{2n}$.   For example, when $\Omega$ is a triangle with legs on the axes, then $X_{\Omega}$ is an ellipsoid; when $\Omega$ is a rectangle with legs on the axes, then $X_{\Omega}$ is a polydisc.

A toric domain $X_{\Omega}$ is called a {\em convex toric domain} if $\Omega$ is a convex connected open subset of the first quadrant containing the origin, and is called {\em rational} if $\Omega$ has rational vertices.  We can define the ellipsoid embedding function $c_{\Omega}(a)$ for any convex toric domain analogously to the definition of $c_b(a)$.

Now recall that a convex polygon with integral vertices is called {\em reflexive} if its dual polygon is also integral.  It is is known that this is equivalent to the triangle having one interior lattice point.

We can now state the conjecture introduced at the beginning of this section:

\begin{conjecture}[\cite{tara}]
\label{conj:hard}
The embedding function $c_\Omega(a)$ of a rational convex toric domain has infinitely many singular points only if some scaling of $\Omega$ is reflexive.
\end{conjecture}

An integral triangle with vertices $(m,0), (0,0),$ and $(0,n)$ and $m \ge n$ is reflexive if and only if
\[ (m,n) \in \lbrace (3,2), (4,2), (3,3) \rbrace.\]
Indeed, if $n = 1$, then the triangle has no interior lattice points at all; if $n \ge 3$, then the triangle contains the $(3,3)$ triangle, which has exactly one interior lattice point, so there are too many interior lattice points unless $n = m = 3$; and if $n = 2$, there are no interior lattice points if $m = 2$, and too many if $m > 4$. 

In particular, our main Theorem~\ref{thm:main} therefore implies the following corollary.

\begin{corollary}
Conjecture~\ref{conj:hard} holds for four-dimensional ellipsoids. 

\end{corollary}

\subsection{Acknowledgements}

This paper is an offshoot of my joint work with Tara Holm, Alessia Mandini, and Ana Rita Pires \cite{tara}, part of which was summarized in \S\ref{sec:reflex}, which is aimed at understanding the phenomenon of infinite staircases in considerably generality.  I thank my wonderful collaborators for many stimulating discussions.  I also thank Roger Cassals and Renato Vianna for explaining their beautiful work \cite{rr} to me.  

I am extremely grateful to the Institute for Advanced Study, the Minerva Research Foundation, and the NSF, under agreement DMS 1711976, for their support.  

\section{Proof of the main theorem}

We now explain the proof of the main theorem.

\subsection{Outline of the argument}
\label{sec:outline}

We beginning by explaining the basic idea behind the argument.  

It is already known that for fixed $b$, if $a$ is sufficiently large then the function $c_b(a)$ is given by the volume obstruction, by \cite[Thm. 1.3]{hind}.  It was also recently proved in \cite[Prop. 2.1]{tara} that away from a limit of distinct singular points, $c_b(a)$ is piecewise linear.  So, we only have to understand whether or not infinitely many singular points can occur.

 In \S\ref{sec:where} we apply a recent theorem by the author and Holm, Mandini, and Pires to find a unique point $a_0$, determined by $b$, where singular points must accumulate if infinitely many of them exists.  Next, we show in \S\ref{sec:small} and \S\ref{sec:bound} that for all but $4$ values of $b$, this number $a_0$ is small enough that one can understand enough about $c_b(a)$ for $1 \le a \le a_0+ \epsilon $ to rule out the possibility of infinitely many singular points around $a_0$.  The part of the argument in \S\ref{sec:bound} uses the theory of ``embedded contact homology" (ECH) capacities, which we explain there, while the part of the argument in \S\ref{sec:small} is completely elementary.

Three of the four possible values for $b$ from above correspond to the $1, 2,$ and $3/2$ cases, where an infinite staircase in fact exists.  The fourth value corresponds to $b = 4/3$; this turns out to be a delicate and interesting case, which we treat separately in \S\ref{sec:43}; our proof here also uses ECH capacities, together with a powerful theorem by McDuff \cite{m} stating that these capacities give sharp obstructions to ellipsoid embeddings.  The proof of Theorem~\ref{thm:main} is then given in \S\ref{sec:proof}.

\subsection{Computing the accumulation point}
\label{sec:where}

In \cite{tara}, the author and collaborators show that for a large class of symplectic $4$-manifolds, any infinite staircase must accumulate at a unique point characterized as a solution to a certain quadratic equation.   We will want to use these results here, to find this accumulation point.  We begin by summarizing the relevant mathematics, in the special case of ellipsoids.

Any rational ellipsoid $E(1,p/q)$ has a {\em negative weight sequence} 
\[ (w; w_1,\ldots,w_k),\] 
defined by the procedure in \cite[\S 2]{concaveconvex}.  The weights can be read off from the triangle $\Delta_{1,p/q}$, with vertices $(0,0), (1,0)$ and $(0,p/q);$ this should be regarded as the ``moment polytope" of the ellipsoid.   

More precisely, the number $w$ is the smallest real number such that $\Delta_{1,p/q} \subset \Delta_{w,w}$: so, in this case, we have $w = p/q$.  To find the $w_i$, we look at the complement of $\Delta_{1,p/q}$ in $\Delta_{p/q,p/q}$.  This is itself a triangle, which is affine equivalent to a right triangle $\Delta^{(1)}$ with legs on the axes.  The $w_i$ are then given as follows.  We take $w_1$ to be the largest number such that $\Delta_{w_1,w_1} \subset \Delta^{(1)}$; then, if this inclusion is not surjective, we look at the complement of $\Delta_{w_1,w_1}$ in $\Delta^{(1)}$, which is itself a triangle affine equivalent to a right triangle $\Delta^{(2)}$ with legs on the axes; we then take $w_2$ to be the largest number such that $\Delta_{w_2,w_2} \subset \Delta^{(2)}$ and iterate until the complement of $\Delta_{w_k,w_k}$ in $\Delta^{(k)}$ is empty.  For the details, see \cite{concaveconvex}.

We remark that the $w_1, \ldots w_k$ as described above are also called the {\em weight sequence} of the triangle $\Delta^{(1)}$.    

We now define
\[ per(E(1,p/q)) = 3w - \sum w_i, \quad vol(E(1,p/q)) = w^2 - \sum w^2_i ,\]
where $(w;w_1,\ldots,w_k)$ is the negative weight sequence.  The term $vol(E(1,p/q))$, which we denote by $vol$ for short, is the volume of $E(1,p/q),$ appropriately normalized; the term $per(E(1,p/q))$, which we denote by $per$, should be regarded as the perimeter.

We now have the following, from \cite{tara}.

\begin{theorem}(\cite[Thm. 1.10]{tara}, in the special case of an ellipsoid)
\label{thm:key}
Let $b$ be a rational number.  Then, if the ellipsoid embedding function $c_b(a)$ has infinitely many singular points, they must accumulate at $a_0$, the unique solution to
\begin{equation}
\label{eqn:keyequation}
a^2 - \left(\frac{per^2}{vol} - 2\right) a + 1 = 0,
\end{equation}
with $a_0 \ge 1$.  Moreover, $c_b(a) = \sqrt{a/b}.$
\end{theorem}

The above theorem can be used here to prove the following key lemma.

Recall that the $a$ values for the Fibonacci staircase terminated at $a = \tau^4$.  We now define an analogue of $\tau^4$ that varies with $b$.  Assume now that $b=k/l$.  The analogue of $\tau^4$ is defined implicitly by the following lemma.

\begin{lemma}
\label{prop:keyprop}
Fix $b=k/l$.  Then, if the graph of $c(a,b)$ has infinitely many nonsmooth points, they must accumulate at 
\[ a_0 = \frac{k}{l}\left(\frac{k+l+1+\sqrt{(k+l+1)^2 - 4kl}}{2k}\right)^2,\]
and $c_b(a_0) = \sqrt{a_0/b}.$
\end{lemma}

For the benefit of the reader, we connect with the Fibonacci staircase by noting that if $k=l=1$, then 
\[ \frac{k}{l}\left(\frac{k+l+1+\sqrt{(k+l+1)^2 - 4kl}}{2k}\right)^2=\tau^4.\]
We will call $\frac{k}{l}(\frac{k+l+1+\sqrt{(k+l+1)^2 - 4kl}}{2k})^2$ the {\em accumulation point}.

\begin{proof}

By Theorem~\ref{thm:key}, the accumulation point $a_0$ must occur at the unique solution to \eqref{eqn:keyequation} that is at least one, and we must have $c_b(a_0) = \sqrt{a_0/b}.$  

To compute $a_0$ explicitly, we need to compute the terms $per$ and $vol$.  We already computed above that $w = k/l.$  Next, we compute 
\[ \Delta^{(1)} = \Delta_{k/l - 1, k/l}.\]
As mentioned above, the remaining weights $w_i$ can be interpreted as the weight sequence for $\Delta_{k/l-1,k/l}$.  We now apply a result of McDuff-Schlenk from \cite{ms}; specifically, in \cite[Lem. 1.2.6]{ms}, it is shown that for any $\Delta_{1,p/q}$ with $p/q$ in lowest terms, the weight sequence $(a_1,\ldots,a_k)$ satisfies
\[ \sum_i a_i = p/q + 1 - 1/q, \quad \sum_i a^2_i = p/q.\]
In the present situation, then,
we find
\[ per = \frac{3k}{l} - \left(\frac{k}{l} - 1\right) \left( \frac{k}{k-l} + 1 - \frac{1}{k-l}\right) = \frac{k+l+1}{l},\]
and
\[ vol = (k/l)^2 - \left(\frac{k}{l} -1\right)^2\frac{k/l}{k/l - 1} = \frac{k}{l}.\]

It is now convenient to use another version of \eqref{eqn:keyequation}.  That is, it is shown in \cite{tara} that the solutions to \eqref{eqn:keyequation} are the same as the solutions to
\[ a + 1 - \sqrt{a \cdot \frac{per^2}{vol}} = 0.\]
Plugging in for $per$ and $vol$ from above, we therefore get
\[ a + 1 - (k+l+1) \sqrt{ \frac{a}{kl} } = 0.\]
Thus, we see that $a' = \frac{l}{k} a $ satisfies
\[ k a' - (k+l+1)\sqrt{a'} + l = 0,\]
hence the result.

\end{proof}

\subsection{The accumulation point is usually small}
\label{sec:small}

We now collect some elementary arguments to show that for most $k$ and $l$, 
\[\frac{k}{l}\left(\frac{k+l+1+\sqrt{(k+l+1)^2 - 4kl}}{2k}\right)^2,\]
is quite small.

\begin{lemma}
\label{lem:nicebound}  Assume that $l \ne 1$, and assume that $(k,l) \not \in \lbrace(3,2), (5,2), (4,3), (5,3), (5,4) \rbrace$.  Then
\[ \frac{k}{l}\left(\frac{k+l+1+\sqrt{(k+l+1)^2 - 4kl}}{2k}\right)^2 < \frac{k+l+1}{l}.\]
\end{lemma}
\begin{proof}  

{\em Step 1.}  Here we prove the following claim.

\begin{claim}
\label{clm:theclaim}
If $l \ge 7$ and $k \ne l$, then $(k+l+1)^2 - 4kl \le (k-l/4-2/5)^2$.
\end{claim}

\begin{proof}[Proof of claim]
We know that
\[ (k-l/4-2/5)^2 = k^2-kl/2-4k/5+l/5+l^2/16+4/25.\]
We also know that
\[ (k+l+1)^2 - 4kl = k^2+l^2+1-2kl+2k+2l.\]
Hence, the claim is true if and only if
\[ \frac{15}{16}l^2-k\left(\frac{3}{2}l-14/5\right)+\frac{9}{5}l + 21/25 \le 0.\]
We know that $\frac{3}{2}l - \frac{14}{5} > 0$ (since $l \ge 2$).  We also know that  $k \ge l+1$.  Hence, we know that 
\begin{align*}
\frac{15}{16}l^2-k\left(\frac{3}{2}l-14/5\right)+\frac{9}{5}l + 21/25 & \le \frac{15}{16}l^2-(l+1)\left(\frac{3}{2}l-14/5\right)+\frac{9}{5}l + 21/25 \\
                                                      & = -\frac{9}{16} l^2 + \frac{31}{10}l+91/25. 
\end{align*}

The larger of the two roots of $-\frac{9}{16} l^2 + \frac{31}{10}l+91/25$ is smaller than seven.  So, since if $l \ge 7,$ 
\[ -\frac{9}{16} l^2 + \frac{31}{10}l+91/25 < 0,\]
the result follows.

\end{proof}

{\em Step 2.}  Claim~\ref{clm:theclaim} is very useful when $l \ge 7$.  We need a slightly different version of this claim to handle most of the other $l$.

\begin{claim}
\label{clm:theclaim1}
If $l \ge 3$ and $k \ge l+6$, then $(k+l+1)^2 - 4kl \le (k-l/4-2/5)^2$.
\end{claim}

\begin{proof}[Proof of claim]
From the proof of Claim~\ref{clm:theclaim}, we know that Claim~\ref{clm:theclaim1} is true if and only if 
\[\frac{15}{16}l^2-k\left(\frac{3}{2}l-14/5\right)+\frac{9}{5}l + 21/25 \le 0.\]
We know that $(\frac{3}{2} l - 14/5) > 0$.  We also know that $k \ge l+6.$  Hence
\begin{align*}
\frac{15}{16}l^2-k\left(\frac{3}{2}l-14/5\right)+\frac{9}{5}l + 21/25 & \le \frac{15}{16}l^2-(l+6)\left(\frac{3}{2}l-14/5\right)+\frac{9}{5}l + 21/25 \\
                                                      & = \frac{-1}{400}\left(225l^2+1760l-7056\right). 
\end{align*}
Since if $l \ge 3,$ 
\[ 225l^2+1760l-7056 > 0,\]
the result follows.

\end{proof}

{\em Step 3.}  Using these two claims, we can now take care of almost every case. 

More precisely, in this step, assume that either $l \ge 7$, or $l \ge 3$ and $k \ge l+ 6$.
Then by Claim~\ref{clm:theclaim} and Claim~\ref{clm:theclaim1}, we know that
\begin{align*}
\frac{k}{l}\left(\frac{k+l+1+\sqrt{(k+l+1)^2 - 4kl}}{2k}\right)^2 & \le \frac{k}{l}\left(\frac{k+l+1+\sqrt{(k-l/4-2/5)^2}}{2k}\right)^2 \\
                                                                                & = \frac{(2k+\frac{3}{4}l+3/5)^2}{4kl} \\
                                                                                & = \frac{4k^2+3kl+\frac{12}{5}k+\frac{9}{16}l^2+9/25+\frac{9}{10}l}{4kl} \\
                                                                                & = \frac{1}{l}\left(k+3l/4+12/20 + \frac{9}{64} \frac{l}{k}l+\frac{9}{60k}+\frac{9l}{40k}\right).
\end{align*}

We know that $k \ge 1,$ and $l/k \le 1$.  Hence, we know that 
\[ \left(k+3l/4+12/20 + \frac{9}{64} \frac{l}{k}l+\frac{9}{60k}+\frac{9l}{40k}\right) \le k +l + 1.\]

This completes the proof of Lemma~\ref{lem:nicebound} in the case where $l \ge 7$, or $l \ge 3$ and $k \ge l +6$.

{\em Step 4.}  Now assume that $l=2$ and let $k \ge 8$; we will now prove Lemma~\ref{lem:nicebound} in this case.

As $k \ge 6$, we know that
\[ k^2 -2k + 9 < (k-1/4)^2.\]
We therefore know that
\begin{align*} \frac{k}{l}\left(\frac{k+l+1+\sqrt{(k+l+1)^2 - 4kl}}{2k}\right)^2 & = \frac{k}{l}\left(\frac{k+3+\sqrt{(k+3)^2 - 8k}}{2k}\right)^2 \\
                                                                                 &  \le \frac{1}{l} \frac{(k + 3 + k - 1/4)^2}{4k} \\
                                                                                 & = \frac{1}{l} \frac{4k^2 + 11k + 121/16}{4k} \\
                                                                                 & = \frac{1}{l}\left(k+11/4 + \frac{121}{64k}\right) \\
                                                                                 & \le \frac{1}{l} (k + 3),
\end{align*}
where, in the last inequality, we have used the fact that $k \ge 8$.  Thus, Lemma~\ref{lem:nicebound} holds in this case as well.

{\em Step 5.}  The previous steps have proved Lemma~\ref{lem:nicebound} under the assumption that $l \ge 7$, or $l \ge 3$ and $k \ge l+6$, or $l = 2$ and $k \ge 8$. 

Thus, it remains to check Lemma~\ref{lem:nicebound} in the following cases: 
\begin{align*}
(k,l) & \in \lbrace (7,2), (7,3),(8,3),(7,4),(9,4), \\
& (6,5),(7,5),(8,5),(9,5),(7,6),(11,6) \rbrace.
\end{align*}
We can compute directly that Lemma~\ref{lem:nicebound} holds for these as well.   
\end{proof}

\subsubsection{Rounding up the non-integral stragglers}

We can deal with the $(5,2), (5,3)$ and $(5,4)$ cases
by using the following simple fact:

\begin{claim}
\label{clm:exceptionalclaim}
If $b=(k,l) \in \lbrace (5,2), (5,3), (5,4) \rbrace$, then 
\[ \frac{k}{l}\left(\frac{k+l+1+\sqrt{(k+l+1)^2 - 4kl}}{2k}\right)^2 < b(\lfloor b \rfloor + 2)^2/(\lfloor b \rfloor + 1)^2.\]
\end{claim}

\begin{proof}  This is verified by direct computation.
\end{proof}

\subsubsection{The integral case}

It is easy to see that Lemma~\ref{lem:nicebound} is not true when $l = 1$.  However, the following is true:

\begin{lemma}
\label{lem:niceintegralbound}
Let $k \ge 3$ and let $l = 1$. 
Then 
\[  \frac{k}{l}\left(\frac{k+l+1+\sqrt{(k+l+1)^2 - 4kl}}{2k}\right)^2 < k(k+3)^2/(k+1)^2.\]
\end{lemma}

\begin{proof}

First, we have:
If $k \ge 4$, then
\[ (k+2)^2 - 4k < (k+1/2)^2 .\]

We now show that this implies Lemma~\ref{lem:niceintegralbound} for $k \ge 4$.  Indeed, in this case we have
\begin{align*}
\frac{k}{l}\left(\frac{k+l+1+\sqrt{(k+l+1)^2 - 4kl}}{2k}\right)^2 & = k\left(\frac{k+2+\sqrt{(k+2)^2 - 4k}}{2k}\right)^2 \\
         & < k\left(\frac{k+2+k+1/2}{2k}\right)^2  \\
         & = k\left(\frac{k+5/4}{k}\right)^2
\end{align*}
Since 
\[ \frac{k + 5/4}{k} < \frac{k+3}{k+1},
\]
if $k \ge 4$ (in fact, even if $k \ge 2$),
the result follows in this case.

Thus, we need only consider the case where $k = 3$.  But this can be verified by direct computation.

\end{proof}

\subsection{Bounding the graph of $c_b(a)$ from below}
\label{sec:bound}

The aim of this section is to prove the following lemma, which will make use of the estimates on the accumulation point from \S\ref{sec:small}.

\begin{lemma}
\label{lem:punchy}
Assume that $(k,l) \not \in \lbrace (1,1), (2,1), (3,2), (4,3) \rbrace$.  Let 
\[ a \le \frac{k}{l}\left(\frac{k+l+1+\sqrt{(k+l+1)^2 - 4kl}}{2k}\right)^2\]
and assume that $c_b(a)$ is equal to the volume obstruction.  Then 
\begin{equation}
\label{eqn:ineq1}
c_b(x) \ge c_b(a)
\end{equation}
for $x \le a$ sufficiently close to $a$, and 
\begin{equation}
\label{eqn:ineq2} 
c_b(x) \le \frac{x}{a} c_b(a)
\end{equation}
for $a \le x$ close to $a$.  
\end{lemma}

To motivate for the reader why this lemma will be useful for us, we remark that we will later show that the inequalities \eqref{eqn:ineq1} and \eqref{eqn:ineq2} can be upgraded to very useful equalities under the assumptions of the lemma, using some general properties of the function $c_b(a)$; we defer this short argument to later in the paper, focusing on the obstructive theory in this section.  

The proof of Lemma~\ref{lem:punchy} will use the theory of ``ECH capacities", defined in \cite{qech}.  The ECH capacities of a symplectic $4$-manifold $(X,\omega)$ are a sequence of nonnegative real numbers
\[ 0 \le c_0(X,\omega) \le \ldots \le c_k(X,\omega) \le \ldots \le \infty\]
that are monotone with respect to symplectic embeddings.  That is, if there is a symplectic embedding
\[ (X_1,\omega_1) \to (X_2,\omega_2),\]
then we must have
\begin{equation}
\label{eqn:obstruct}
c_k(X_1,\omega_1) \le c_k(X_2,\omega_2),
\end{equation}
for all $k$.  Hence, ECH capacities are obstructions to the existence of a symplectic embedding.  ECH capacities are defined using ``embedded contact homology"; for more, see for example the survey article \cite{pnas}.  

In the case of ellipsoids, the ECH capacities have been computed in \cite{qech}.  The result is that $c_k(E(a,b))$ is the $(k+1)^{st}$ smallest element in the matrix
\[(ma+nb)_{ (m, n) \in \mathbb{Z}_{\ge 0} \times \mathbb{Z}_{\ge 0}}.\]

Using ECH capacities, we can now prove the following lower bound, which we will then use to prove Lemma~\ref{lem:punchy}.

\begin{lemma}
\label{lem:intlemma}
Fix any real number $b \ge 1$.  Then:
\begin{itemize}
\item $c_b(a) = 1$ for $1 \le a \le b$.
\item $c_b(a) \ge a/b$ for $b \le a \le \lfloor b \rfloor +1$.
\item $c_b(a) \ge (\lfloor b \rfloor+1)/b$ for $\lfloor b \rfloor+1 \le a \le (\lfloor b \rfloor +1)^2/b.$
\item $c_b(a) \ge a/(\lfloor b \rfloor+1)$ for $(\lfloor b \rfloor +1)^2/b \le a \le \lfloor b \rfloor+2.$
\item $c_b(a) \ge (\lfloor b \rfloor + 2)/(\lfloor b \rfloor + 1)$ for $\lfloor b \rfloor + 2 \le a \le b(\lfloor b \rfloor + 2)^2/(\lfloor b \rfloor + 1)^2.$ 
\item If $b$ is an integer, then $c_b(a) \ge a/(b+1)$ for $(b+1)^2/b \le a \le b+3.$
\item If $b$ is an integer, then $c_b(a) \ge (b+3)/(b+1)$ for $b+3 \le a \le b(b+3)^2/(b+1)^2.$  
\end{itemize}
\end{lemma}  

Concerning the statement of the lemma, we remark, for example, that it might be the case that various bullets points are vacuously true --- for example, for $b = 1.2$, $(\lfloor b \rfloor +1)^2/b > \lfloor b \rfloor+2$.

\begin{proof}

To prove the first bullet point, we note that $E(1,a)$ includes into $E(1,b)$ for $a \le b$; since $c_1(E(1,a)) = c_1(E(1,b)) = 1$, this inclusion is optimal by \eqref{eqn:obstruct}, in the sense that no larger scaling of $E(1,a)$ also embeds, so the bullet point holds.

To prove the second and third bullet points, 
we note first that $c_{\lfloor b \rfloor + 1}(E(1,b)) = b$.  Then, with $b \le a \le \lfloor b \rfloor + 1$, we have $c_{\lfloor b \rfloor +1}(E(1,a)) = a$, so that the second bullet point follows by \eqref{eqn:obstruct}; and with $a \ge \lfloor b \rfloor + 1$, we have  $c_{\lfloor b \rfloor + 1}(E(1,a)) = \lfloor b \rfloor + 1$, hence the third bullet point follows by \eqref{eqn:obstruct}.

The prove the fourth and fifth bullet points, we note first that $c_{\lfloor b \rfloor+2}(E(1,b)) = \lfloor b \rfloor + 1.$ Then, if $a \ge \lfloor b \rfloor + 2$, we have $c_{\lfloor b \rfloor+2}(E(1,a)) = \lfloor b \rfloor + 2,$ hence the fifth bullet point follows by \eqref{eqn:obstruct}.  If $( \lfloor b \rfloor + 1 )^2/b \le a \le \lfloor b \rfloor + 2$, then, as $( \lfloor b \rfloor + 1 )^2/b \ge \lfloor b \rfloor + 1$, we must have  $c_{\lfloor b \rfloor+2}(E(1,a)) = a,$ hence the fourth bullet point follows by \eqref{eqn:obstruct}.

To prove the sixth and seventh bullet points, we note that if $b$ is an integer, then $c_{b+3}(E(1,b)) = b+1$.  Then, if $b+2 \le a \le b+3$, we have $c_{b+3}(E(1,a)) = a,$ hence the sixth bullet point follows by \eqref{eqn:obstruct}, since for $a$ in the domain of the sixth bullet point, $b+2 \le a \le b+3$.  If $a \ge b+3$, we have $c_{b+3}(E(1,a)) = b+3$, hence the seventh bullet point follows by \eqref{eqn:obstruct}.
\end{proof}

We can now prove the main result of this section.  

\begin{proof}[Proof of Lemma~\ref{lem:punchy}]

Recall that the volume obstruction is given by $\sqrt{a/b}$.  We can find the point $a_i$ on the domain of the $i^{th}$ bullet point of Lemma~\ref{lem:intlemma} where the volume obstruction agrees with the lower bound given by each bullet point by setting this lower bound equal to the volume obstruction, and solving for the point $a_i$.  Doing this gives:
\[ a_1 = b, a_2 = b, a_3 = \frac{ (\lfloor b \rfloor + 1)^2 }{b}, a_4 = \frac{ (\lfloor b \rfloor + 1)^2 }{b}, a_5 = b \left(\frac{  \lfloor b \rfloor + 2 }{ \lfloor b \rfloor + 1} \right)^2, \]
\[ a_6 = \frac{ (b+1)^2 }{b}, a_7 = b \left(\frac{b+3}{b+1}\right)^2. \]
We also compute that on each of these intervals, away from $a_i$ the lower bound given by Lemma~\ref{lem:intlemma} is strictly larger than the volume bound.

We next observe that for $a_1, \ldots, a_4, a_6$, it follows from Lemma~\ref{lem:intlemma} that the bounds \eqref{eqn:ineq1} and \eqref{eqn:ineq2} required by Lemma~\ref{lem:punchy} hold.   More precisely, \eqref{eqn:ineq1} and \eqref{eqn:ineq2} for $a_1$ and $a_2$ follow from the first two bullet points; \eqref{eqn:ineq1} and \eqref{eqn:ineq2} for $a_3$ and $a_4$ follow from the third and fourth; for $a_6$, this follows from the third and sixth.

Now assume first that $(k,l)$ is such that the assumptions of Lemma~\ref{lem:nicebound} hold.  Then the accumulation point $a_0$ is bounded by $k/l + 1/l + 1$, which in turn is less than or equal to $\lfloor k/l \rfloor + 2.$   Thus, if $a \le a_0$, then $a$ is in the domain of one of the first four bullet points of Lemma~\ref{lem:intlemma}.  Thus, if $c_b(a) = \sqrt{a/b}$, then $a \in \lbrace a_1, \ldots, a_4 \rbrace$, since at any other point in the $i^{th}$ interval, $c_b$ is bounded from below by a function that is strictly larger than the volume bound.  So, the conclusions of  Lemma~\ref{lem:nicebound} hold in this case by the analysis in the previous paragraph.

Next, assume that $(k,l) \in \lbrace (5,2), (5,3), (5,4) \rbrace$.  Then, by Claim~\ref{clm:exceptionalclaim}, $a_0$ is  strictly bounded from above by $b( \lfloor b \rfloor + 2)^2/( \lfloor b \rfloor + 1)^2,$ for $b = k/l.$   Thus, if $a \le a_0$, then $a$ is in the domain of one of the first five bullet points, but is not the right end point of the fifth and in particular must be strictly smaller than $a_5$; hence if $c_b(a) = \sqrt{a/b}$, then as  in the previous paragraph $a \in \lbrace a_1, \ldots a_4 \rbrace$, so that the conclusions of Lemma~\ref{lem:nicebound} hold in this case as well.

Finally, assume that $l = 1$, and $k \ge 3$.  Then $b = k/l$ is an integer.  By Lemma~\ref{lem:niceintegralbound} the accumulation point $a_0$ is strictly bounded from above by $b(b+3)^2/(b+1)^2.$  Hence, if $a \le a_0$, then $a$ is in the domain of either the first three bullet points, or the sixth or seventh; moreover, it is not the right end point of the seventh and in particular must be strictly smaller than $a_7$.  It follows that if $c_b(a) = \sqrt{a/b}$, then $a \in \lbrace a_1, \ldots, a_3, a_6 \rbrace$, and so just as in the previous paragraphs the conclusions of Lemma~\ref{lem:nicebound} hold as well.

\end{proof}

\subsection{The $E(1,4/3)$ case}
\label{sec:43}

To deal with the case where $b = 4/3$, we need to prove the following.

\begin{proposition}
\label{lem:randomlemma}
For $\epsilon$ sufficiently small,
\[ c(a,4/3) = \frac{a+3}{4}\]
if $3 \le a \le  3 +\epsilon$, and
\[ c(a,4/3) = 3/2,\]
if $3 - \epsilon \le a \le 3.$
\end{proposition}

The proof of Proposition~\ref{lem:randomlemma} is rather delicate, and will be the topic of this section.  The result itself is loosely analogous to the difficult \cite[Thm. 1.1.2.ii]{ms}, although we use a different method in our proof.  The main challenging fact that we need to prove is the following:

\begin{proposition}
We have
\begin{equation}
\label{eqn:upperbound}
c_{4/3}(a) \le \frac{a+3}{4}.
\end{equation}
for $a \ge 3$.
\end{proposition}

\begin{proof}

To prove \eqref{eqn:upperbound}, we want to show that there exists a symplectic embedding
\[ E(1,a) \to \frac{a+3}{4} E(1,4/3).\]
By rescaling, it is equivalent to find an embedding
\begin{equation}
\label{eqn:desiredembedding}
E\left(\frac{12}{a+3}, \frac{12a}{a+3}\right) \to E(3,4).
\end{equation}
Since $c_{4/3}$ is continuous in $a$, we can in addition assume that $a$ is irrational, which is convenient for some of the arguments below.

To find this embedding, we use in general terms a technique first introduced in \cite{cgk, qech, ms, m}. 

Namely, McDuff showed in \cite{m} that the obstruction coming from ECH capacities is in fact sharp for four-dimensional ellipsoid embeddings.  In other words, the existence of embeddings like \eqref{eqn:desiredembedding} can be approached through purely combinatorial considerations.  In principle, since there are infinitely many ECH capacities $c_k$, this requires checking infinitely many potential obstructions.  However, in \cite{cgk}, this was rephrased in rational cases in terms of ``Ehrhart functions", defined below.  Ehrhart functions are a classical object of study in enumerative combinatorics which are often amendable to computations.    

More precisely, we can apply \cite[Lem. 5.2]{cgk} to conclude that an embedding \eqref{eqn:desiredembedding} exists if and only if
\begin{equation}
\label{eqn:latticeproblem}
L_{\mathcal{T}_{\frac{a+3}{12}, \frac{a+3}{12a}}}(t) \ge L_{\mathcal{T}_{\frac{1}{3}, \frac{1}{4}}}(t),
\end{equation}
for all positive integers $t$.  Here, $\mathcal{T}_{u,v}$ denotes\footnote{The paper \cite{cgk} actually uses the convention that $\mathcal{T}_{u,v}$ denotes the triangle with vertices $(0,u), (v,0)$ and $(0,0)$, but this triangle has the same number of lattice points as the triangle defined using the conventions in this paper.} the triangle with vertices $(u,0)$ and $(0,v)$, and $L$ denotes its {\em Ehrhart function} 
\[ L_{\mathcal{T}_{u,v}}(t) = \# \left \lbrace \mathbb{Z}^2 \cap \mathcal{T}_{tu, tv} \right \rbrace.\]

Our method is now loosely inspired by the proof in \cite[Lem. 3.2.3]{ghost}, see also \cite[Rmk. 3.2.6]{ghost}, although there is a new idea needed here that we will comment on below.  

As in the proof in \cite[Lem. 3.2.3]{ghost}, we will first observe that \eqref{eqn:latticeproblem} holds when $a = 3$.  In fact, strict inequality holds in \eqref{eqn:latticeproblem}, as we will see below.  The idea is now to vary $a$ and see how $L_{\mathcal{T}_{\frac{a+3}{12}, \frac{a+3}{12a}}}(t)$ changes.   As in \cite[Lem. 3.2.3]{ghost}, we do this by decomposing the region between $\mathcal{T}_{\frac{a+3}{12}, \frac{a+3}{12a}}$ for some $a$ and $\mathcal{T}_{\frac{3+3}{12}, \frac{3+3}{12 \cdot 3}}$ into two regions $R_U$ and $R_D$, and comparing the number of lattice points $U$ and $D$.

\begin{figure}[h!]
\centering
\includegraphics[width=1.2\textwidth]{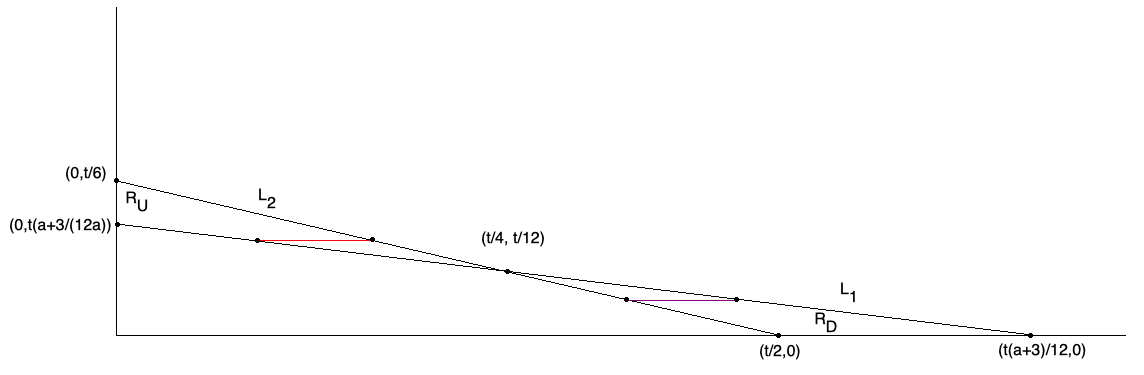}
\caption{The regions $R_U$ and $R_D$ that we want to compare.  The lines $L_1$ and $L_2$ are also labeled for the convenience of the reader.  The red and purple lines correspond to the kind of horizontal slices that we make in order to compare lattice point counts.}
\label{fig:figure}
\end{figure}

More precisely, for positive integer $t$, we let $R_U$ be the region bounded by the $y$-axis, the line $L_1$ given by the equation $\frac{12}{a+3}x + \frac{12a}{a+3}y = t$, and the line $L_2$ given by the equation $2x + 6y = t$.  Let $R_D$ be the region bounded by these two lines, and the $x$-axis.  Let $U$ denote the number of lattice points in $R_U$ and $D$ the number of lattice points in $R_D$.  We note that the lines $L_1$ and $L_2$ intersect at the point $(t/4, t/12).$    We have illustrated the setup in Figure~\ref{fig:figure}.

We now have the following key lemma:

\begin{lemma}
\label{lem:keylem}
\begin{itemize}
\item $U \le D$. 
\item When $t$ is congruent to $4$ modulo $12$, then $U \le D-1$.  
\end{itemize}
\end{lemma}

Up to this point, our method in this section has been mostly parallel to the method in \cite{ghost} described above.  However, at this point, the ideas in \cite{ghost} no longer seem to work, and something new is needed.  The new technique we introduce here is to  compare the lattice points in $R_U$ and $R_D$ by comparing the number of lattice points on horizontal slices at integer height, see Figure~\ref{fig:figure}.  It turns out that we can get the inequality we need by establishing the analogous inequality for each slice individually, which is a priori considerably stronger than what is required.

The details are as follows.

\begin{proof}

We begin with the proof of the first bullet point.

Let $t/6 \ge y_0 \ge t/12$ be an integer.  We define 
\[ y_1 = \lfloor t/6 \rfloor - y_0.\]
Then $0 \le y_1 \le t/12.$  We will show that for each $y_0$, the number of lattice points in $R_U$ with $y$-coordinate $y_0$ is no more than the number of lattice points in $R_D$ with $y$-coordinate $y_1$, which will imply the first bullet point of the lemma.

In other words, 
if we define
\[ x_1 := \frac{t(a+3) - 12ay_0}{12}, \quad x_2 := \frac{t - 6y_0}{2},\]
and
\[ x_3 := \frac{t - 6y_1}{2}, \quad x_4 := \frac{t(a+3) - 12ay_1}{12},\]
then we need to show that 
\begin{equation}
\label{eqn:greatbound}
\lfloor x_2 \rfloor - \lceil max(0,x_1) \rceil + 1 \le \lfloor x_4 \rfloor - \lceil x_3 \rceil + 1.
\end{equation}

We now explain why \eqref{eqn:greatbound} holds.

Our argument will be as follows.  Assume that $x_3$ is not an integer; note that $x_1$ is never an integer, since $a$ is irrational.  
Then, we will show below that
\begin{equation}
\label{eqn:lowerboundversion}
\lfloor x_2 \rfloor - \lfloor x_1 \rfloor \le \lfloor x_4 \rfloor - \lfloor x_3 \rfloor. 
\end{equation}
Next, in the case where $x_3$ is an integer, we will show that
\begin{equation}
\label{eqn:lowerboundversionint}
\lfloor x_2 \rfloor - \lfloor x_1  \rfloor \le \lfloor x_4 \rfloor - \lfloor x_3 \rfloor + 1. 
\end{equation}
The equations \eqref{eqn:lowerboundversion} and \eqref{eqn:lowerboundversionint} will imply \eqref{eqn:greatbound}, since $\lfloor x_2 \rfloor - \lceil max(0,x_1) \rceil + 1 \le \lfloor x_2 \rfloor - \lfloor x_1 \rfloor.$

We now explain why \eqref{eqn:lowerboundversion} and \eqref{eqn:lowerboundversionint} hold.  We know that
\[ \lfloor x_2 \rfloor - \lfloor x_1 \rfloor  = \frac{t}{2} - \left\lbrace \frac{t}{2} \right\rbrace - 3y_0 - \frac{t(a+3)}{12} + ay_0 + \left\lbrace \frac{t(a+3)}{12} - ay_0 \right\rbrace.\]
Here, $\lbrace \cdot \rbrace$ denotes the fractional part function, defined by $\lbrace z \rbrace = z - \lfloor z \rfloor$.
We also know that
\[ \lfloor x_4 \rfloor - \lfloor x_3 \rfloor = \frac{t(a+3)}{12} - a y_1 - \left \lbrace \frac{t(a+3)}{12} - a y_1 \right \rbrace - \frac{t}{2} + \left\lbrace \frac{t}{2} \right\rbrace + 3 y_1.\]

We first prove \eqref{eqn:lowerboundversion} in the case where $x_3$ is not an integer, which is the heart of the argument.

To do this, we want to show, in view of combining the previous two equations, that
\[ t - 2 \lbrace t/2 \rbrace - 3 (y_0+y_1) -  \frac{t(a+3)}{6} + a(y_0+y_1) + \left\lbrace \frac{t (a+3)}{12} - a y_1 \right\rbrace + \left\lbrace  \frac{t(a+3)}{12}  - a y_0 \right\rbrace \le 0.\]
Substituting for $y_1$, we have that the above expression is equal to
\[ t  - 2 \lbrace t/2 \rbrace - 3 \lfloor t/6 \rfloor - \frac{t(a+3)}{6} + a \lfloor t/6 \rfloor + \delta,\]
where 
\[ \delta := \left\lbrace \frac{t(a+3)}{12} - ay_0 \right\rbrace + \left\lbrace \frac{t(a+3)}{12} + ay_0  - a\lfloor t / 6 \rfloor \right\rbrace.\]
So, collecting $\lfloor t/6 \rfloor$ terms, we want to show that
\[ t - 2 \lbrace t/2 \rbrace + (a-3) \lfloor t/6 \rfloor - \frac{t (a+3) }{6} + \delta \le 0. \]
Equivalently, we want to show that
\begin{equation}
\label{eqn:need}
- 2 \lbrace t/2 \rbrace - (a-3) \lbrace t/6 \rbrace + \delta \le 0.
\end{equation}

Since, for any two numbers $m, n$, we have\footnote{Indeed, the equation is invariant under adding integers to $m$ or $n$, so we can assume $0 \le m, n < 1$, in which case it is immediate.} $\lbrace m \rbrace + \lbrace n \rbrace \le \lbrace m + n \rbrace + 1,$ we know that
\begin{equation}
\label{eqn:r1}
\delta \le \left\lbrace \frac{t}{2} + a \lbrace t/6 \rbrace \right\rbrace + 1.
\end{equation}
The terms $\lbrace t/2 \rbrace, \lbrace t/6 \rbrace$ and $\left\lbrace t/2 + a \lbrace t/6 \rbrace \right\rbrace$ only depend on the equivalence class of $t$, modulo $6$.  
So, to bound the left hand side of \eqref{eqn:need} using the above bound for $\delta$, we can assume 
$t \in \lbrace 0, \ldots, 5 \rbrace.$  With this additional assumption, we then have 
\begin{equation}
\label{eqn:r2}
\left\lbrace \frac{t}{2} + a \lbrace t/6 \rbrace \right\rbrace=  \left\lbrace \frac{t}{2} + a \frac{t}{6}  \right\rbrace =   \left\lbrace \frac{(a - 3)t}{6}\right\rbrace \le \frac{(a-3)t}{6} = (a-3) \left\lbrace\frac{t}{6}\right\rbrace.
\end{equation}

Combing \eqref{eqn:r1} and \eqref{eqn:r2}, we thus have that
\begin{equation}
\label{eqn:last}
 - 2 \lbrace t/2 \rbrace - (a-3) \lbrace t/6 \rbrace + \delta \le -2 \lbrace t/2 \rbrace + 1 = 0,
 \end{equation}
where for the very last equality, we have used the fact that $x_3$ is not an integer, so that $t$ is odd.  This proves \eqref{eqn:need}, hence \eqref{eqn:lowerboundversion}.

When $x_3$ is an integer, all of the proof of \eqref{eqn:lowerboundversion} holds, except that in the very last line $\lbrace t/2 \rbrace = 0$, so that the very last equation \eqref{eqn:last} must be replaced by the bound 
\[ - 2 \lbrace t/2 \rbrace - (a-3) \lbrace t/6 \rbrace + \delta \le 1,\]
hence the weaker bound \eqref{eqn:lowerboundversionint}.

We now explain the proof of the second bullet point.

The argument for the first bullet point still holds to imply that $U \le D$.  To get the sharper bound, we show that, under the assumption that $t$ is congruent to $4$ modulo $12$, 
as $y_0$ above ranges over all integers between $t/6$ and $t/12$ the corresponding $y_1$ is never $y' = \lfloor t/12 \rfloor$.  

Indeed, the $y_1$ corresponding to $y_0$ is maximized for $y_0 = \lceil t/12 \rceil$, so in this case $y_1 = \lfloor t/6 \rfloor - \lceil t/12 \rceil.$  Now, 
\[   \lfloor t/6 \rfloor - \lceil t/12 \rceil = \lfloor t/12 \rfloor - 1,\]
since $t$ is congruent to $4$ modulo 12, which is strictly less than $y'$.

Thus, since  $\left( \frac{t - 6y'}{2}, y' \right)$ is a lattice point in $R_D$, not accounted for by the counts in the proof of the first bullet point, the sharper estimate asserted by the second bullet points holds.
\end{proof}

We now explain how to use the lemma to prove the proposition.  Continue to assume as above that $a$ is irrational.  

We first observe that 
\begin{equation}
\label{eqn:diff}
 L_{\mathcal{T}_{\frac{a+3}{12}, \frac{a+3}{12a}}}(t) = L_{\mathcal{T}_{\frac{1}{2}, \frac{1}{6}}}(t) + D - U - d,
 \end{equation}
 where $d$ is the number of lattice points on the left boundary of $D$, not including the possible lattice point $(t/4, t/12)$ defined above.
 
 We can solve for $d$ explicitly.  Namely, assume that there is a lattice point $(m,n)$ satisfying
 \[ 2m + 6n = t.\]
Then, it follows that $t$ must be an even integer.  Conversely, assume that $t$ is an even integer, and $(x,y)$ is on the line $L_2$.  Then we have
\[ x = \frac{t - 6y}{2}.\]
In particular, for any integer $y < t/12$ such that $(x,y)$ is on the line $L_2$, $x$ must be an integer as well.  It follows that
\begin{equation}
\label{eqn:d}
d = \lceil t/12 \rceil,
\end{equation}
when $t$ is even;
if $t$ is odd then we have $d = 0$.

We know from Lemma~\ref{lem:keylem} that $D \ge U$; however, the $-d$ term is not in general non-negative, so to prove \eqref{eqn:latticeproblem}, we need to compute the difference
\[  L_{\mathcal{T}_{\frac{1}{2}, \frac{1}{6}}}(t) -  L_{\mathcal{T}_{\frac{1}{3}, \frac{1}{4}}}(t).\]
Each of the two-terms in the above expression are Ehrhart functions of rational triangles, so they are readily computed.  In particular, using the formulas in \cite[Thm. 2.10, Exer. 2.34]{br} 
each is a periodic polynomial of degree $2$, with leading order term $t^2/24$.  The linear term for $L_{\mathcal{T}_{\frac{1}{3}, \frac{1}{4}}}(t)$, by \cite[Thm. 2.10]{br} is $t/3$.   The linear term for $L_{\mathcal{T}_{\frac{1}{2}, \frac{1}{6}}}(t)$ is $\frac{5}{12} t$, when $t$ is even, and $t/3$, when $t$ is odd, by \cite[Exer. 2.34]{br}.  

To compute the constant terms, we use the fact that the period of $L_{\mathcal{T}_{\frac{1}{2}, \frac{1}{6}}}(t)$ divides $6$, and the period of $L_{\mathcal{T}_{\frac{1}{3}, \frac{1}{4}}}(t)$ divides $12$; indeed, the basic structure theorem for Ehrhart functions (see for example \cite[Thm. 3.23]{br}) states that the period for a rational convex polytope divides the least common multiple of the denominators of the vertices.  

More precisely, we first compute the constant terms for $L_{\mathcal{T}_{\frac{1}{2}, \frac{1}{6}}}(t)$.  We begin by computing,  
\[ L_{\mathcal{T}_{\frac{1}{2}, \frac{1}{6}}}(0) =  L_{\mathcal{T}_{\frac{1}{2}, \frac{1}{6}}}(1)= 1, \]
\[ L_{\mathcal{T}_{\frac{1}{2}, \frac{1}{6}}}(2) = L_{\mathcal{T}_{\frac{1}{2}, \frac{1}{6}}}(3) = 2,\]
\[ L_{\mathcal{T}_{\frac{1}{2}, \frac{1}{6}}}(4)  = L_{\mathcal{T}_{\frac{1}{2}, \frac{1}{6}}}(5) = 3.\]
Now, since $L_{\mathcal{T}_{\frac{1}{2}, \frac{1}{6}}}(t)$ is a periodic polynomial, with period dividing $6$, we can define $C_0, \ldots, C_5$ to be the constant terms for this periodic polynomial, i.e. $C_i$ is the constant term when $t$ is congruent to $i$, modulo $6$.  We can  then compute the constant terms by using the computations above, namely
\[ C_0 = 1 - \frac{1}{24} (0)^2 - \frac{5}{12} (0) = 1,\]
\[ C_1 = 1 - \frac{1}{24} (1)^2 -  \frac{1}{3} (1) = \frac{5}{8},\]
\[ C_2 = 2 - \frac{1}{24} (2)^2 - \frac{5}{12} (2) =  1,\]
\[ C_3 = 2 - \frac{1}{24} (3)^2 - \frac{1}{3} (3) = 5/8,\]
\[ C_4 = 3 - \frac{1}{24} (4)^2 - \frac{5}{12} (4) = 2/3,\]
\[ C_5 = 3 - \frac{1}{24} (5)^2 - \frac{1}{3} (5) = 7/24.\]

We can compute the constant terms $C'_0, \ldots, C'_{11}$  by the same method.  We omit the details, which are analogous to above, for brevity, only giving the result:
\[ C'_0 = 1, C'_1 = 5/8, C'_2 = 1/6, C'_3 = 5/8,  C'_4 = 1, C'_5 = 7/24, C'_6 = 1/2,\]
\[ C'_7 = 5/8, C'_8 = 2/3, C'_9 = 5/8, C'_{10} = 1/2, C'_{11} = 7/24.  \]

Having computed both Ehrhart functions explicitly, and applying the formula \eqref{eqn:d} for $d$, we now see that
\[ L_{\mathcal{T}_{\frac{1}{2}, \frac{1}{6}}}(t) -  L_{\mathcal{T}_{\frac{1}{3}, \frac{1}{4}}}(t) = d ,\]
except when $t$ is congruent to $4$ mod $12$, in which case the difference in the Ehrhart functions is $d - 1$.  The proposition now follows from \eqref{eqn:diff}, in combination with Lemma~\ref{lem:keylem}.  
\end{proof}

We can now prove Proposition~\ref{lem:randomlemma}.

\begin{proof}

We just showed that $c_{4/3}(a) \le \frac{a+3}{4}$ for $a \ge 3$.  In particular, as an immediate consequence, $c_{4/3}(3) = 3/2$, since a symplectic embedding must be volume preserving, and then 
\begin{equation}
\label{eqn:rlowerbound}
c_{4/3}(a) \le 3/2, 
\end{equation}
for $a \le 3$, since $E(1,a) \subset E(1,3)$ for $a$ in this range.

To find the lower bounds needed to prove the proposition, we again use the theory of ECH capacities.  

More precisely, we first compute 
\[ c_{10}(E(1,4/3)) = 4, \quad c_{10}(E(1,a)) = a+3,\]
 for $3 \le a \le 4.$  Hence, $c_{4/3}(a) \ge \frac{a+3}{4}$ for $3 \le a \le 4$, by \eqref{eqn:obstruct}.  Combining this with the matching upper bound \eqref{eqn:upperbound} then implies that $c_{4/3}(a) = \frac{a+3}{4}$ for $a$ in this range.

We next compute
\[ c_{2}(E(1,4/3)) = 4/3, \quad c_{2}(E(1,a)) = 2,\]
for $a \ge 2$.  Hence, $c_{4/3}(a) \ge 3/2$ for $2 \le a \le 3$, by \eqref{eqn:obstruct}.  Combining this with the matching upper bound \eqref{eqn:rlowerbound} then implies that $c_{4/3}(a) = \frac{3}{2}$ for $a$ in this range.
\end{proof}
\subsection{Completing the proof of Theorem~\ref{thm:main}}
\label{sec:proof}

We can now complete the proof of our main theorem.  

\begin{proof}[Proof of Theorem~\ref{thm:main}]

As explained in \S\ref{sec:outline}, it follows from known results that the function $c_b(a) = \sqrt{a/b}$ for $a$ sufficiently large with respect to $b$; it also follows from known results that the function $c_b(a)$ is piecewise linear away from the limit of distinct singular points.  So, we just have to analyze the case of infinitely many distinct singular points.

Let $b=k/l$, where $k$ and $l$ are relatively prime, and recall the number
\[a_0 =  \frac{k}{l}\left(\frac{k+l+1+\sqrt{(k+l+1)^2 - 4kl}}{2k}\right)^2\]
from Lemma~\ref{prop:keyprop}.

Assume that there are infinitely many singular points $s_i$.  Then, by Lemma~\ref{prop:keyprop}, the $s_i$ must accumulate at $a_0$, and $c_b(a_0)$ must equal the volume obstruction.  We now argue that there is a contradiction  if $b \not \in \lbrace 1, 2, 3/2 \rbrace.$

 Namely, if $b \not \in \lbrace 1, 2, 3/2, 4/3 \rbrace$, then we know from Lemma~\ref{lem:punchy} that \eqref{eqn:ineq1} and \eqref{eqn:ineq2} hold at $a_0$.

We now claim that this implies that the graph of $c_b(a_0)$ would locally be given by these lines near $a_{0}$ --- which is an absurdity, since $a_0$ was the limit of distinct singular points. 

To see why this final claim is true, we need the following two properties for the function $c_b(a)$:
\begin{itemize}
\item (Monotonicity) $c_b(x_1) \le c_b(x_2)$ if $x_1 \le x_2$.
\item (Subscaling) $c_b(\ell x_1) \le \ell c_b(x_1)$ for any $\ell \ge 1$.
\end{itemize}
The first bullet point is immediate, since $E(1,x_1) \subset E(1,x_2)$ if $x_1 \le x_2$.  The second follows by a short scaling argument, see for example \cite[Prop. 2.1]{tara} for the details.  

With these two properties, we can now verify the final claim --- in view of the lower bounds \eqref{eqn:ineq1} and \eqref{eqn:ineq2}, monotonicity would then imply that \eqref{eqn:ineq1} is an equality for $x \le a_0$ close to $a_0$, and subscaling would then imply that \eqref{eqn:ineq2} is an equality for $x \ge a_0$ close to $a_0$.  

If $b = 4/3$, then it follows from Proposition~\ref{lem:randomlemma} that $c_b(a_0)$ has a unique singular point near $a_0$, namely $a_0$ itself.  Thus, in this case $a_0$ also can not be the limit of distinct singular points.

\end{proof}

\end{document}